\documentclass[reqno,final,a4paper]{amsart}
\usepackage{color}
\usepackage[colorlinks=true,allcolors=blue]{hyperref}
\usepackage{mathrsfs}
\usepackage{mathtools}
\usepackage[linesnumbered,ruled,vlined]{algorithm2e}

\SetCommentSty{mycommfont}
\SetKwInOut{Input}{input}\SetKwInOut{Output}{output}

\usepackage[noadjust]{cite}
\usepackage[noabbrev,capitalize,nameinlink]{cleveref}
\crefname{equation}{}{}
\usepackage{fullpage}
\usepackage{graphics}
\usepackage{tikzscale}
\usepackage{pifont}
\usepackage{tikz}
\usetikzlibrary{fit, backgrounds, shapes, calc}
\usepackage{pgfplots}
\pgfplotsset{compat=1.12}
\usepackage{bbm}
\usepackage[T1]{fontenc}
\usetikzlibrary{arrows.meta}

\usepackage{environ}
\usepackage{framed}
\usepackage{url}
\usepackage{thm-restate}

\usepackage[noend]{algpseudocode}
\usepackage[labelfont=bf]{caption}
\usepackage{framed}
\usepackage[framemethod=tikz]{mdframed}
\usepackage{appendix}
\usepackage{graphicx}
\usepackage[textsize=tiny]{todonotes}
\usepackage{tcolorbox}
\usepackage{xcolor}
\usepackage[shortlabels]{enumitem}
\crefformat{enumi}{#2#1#3}
\crefrangeformat{enumi}{#3#1#4 to~#5#2#6}
\crefmultiformat{enumi}{#2#1#3}
{ and~#2#1#3}{, #2#1#3}{ and~#2#1#3}
\allowdisplaybreaks

\usepackage{hyperref}
\usepackage{nameref}

\makeatletter
\newcommand{\namedlabel}[2]{%
    \phantomsection
    \def\@currentlabelname{#2}%
    \label{#1}%
}
\makeatother

\usepackage{ifthen}
\numberwithin{equation}{section}

\newtheorem{theorem}{Theorem}[section]

\newtheorem{lemma}[theorem]{Lemma}

\crefname{claim}{Claim}{Claims}

\newtheorem{conjecture}[theorem]{Conjecture}
\newtheorem*{question*}{Question}

\theoremstyle{definition}

\newtheorem*{definition*}{Definition}

\newtheorem*{fact*}{Fact}

\crefname{fact}{Fact}{Facts}

\theoremstyle{remark}

\DeclareMathOperator{\Bin}{Bin}

\renewcommand{\Pr}{\mathbb P}

\newlength{\entropynamewidth}

\title{Quadratic-size Random Permutations are Universal}
\author{Yu Chen}
\address{(Y. Chen) School~of~Mathematics~and~Statistics,~Beijing~Institute~of~Technology,~Beijing,~102488,~China.}\email{yu.chen2023@bit.edu.cn}

\author{Yiting Wang}
\address{(Y. Wang) Institute of Science and Technology Austria,~Klosterneuburg,~3400,~Austria.}
\email{yiting.wang@ist.ac.at}

\thanks{The second author is supported by the European Research Council (ERC), via grant agreements ``RANDSTRUCT'' No.\ 101076777.}
\date{\today}

\begin{document}

\begin{abstract}
A permutation is $n$-universal if it contains every permutation of length $n$ as a pattern. We prove that a uniformly random permutation of length $32n^2$ is $n$-universal with probability at least $1-\exp(-c n/\log n)$, for an absolute constant $c>0$ and all sufficiently large $n$. This removes a $\log\log n$ factor from an earlier result of He and Kwan~\cite{He-Kwan} and resolves an open problem of Alon~\cite{Alon2016}.
The proof was found by GPT-6 Pro.
\end{abstract}

\maketitle

\section*{Declaration on the usage of generative AI}
The proof was found by GPT-6 Pro.
The authors do not claim credit for the proof ideas.
They have verified and digested the proof and take full responsibility for the correctness of the manuscript. 

\section{Introduction}
Universality asks how small or sparse a combinatorial object can be while
containing every member of a prescribed family. Classical graph examples
include Rado's universal graphs~\cite{Rado1964} and induced universal
graphs for all $n$-vertex graphs~\cite{Moon1965,Alon2017}, bounded-degree
graphs~\cite{AlonNenadov2019}, and planar graphs~\cite{EsperetJoretMorin2023}.
For ordinary subgraph containment, sparse universal graphs have been
constructed for families with bounded maximum degree~\cite{AlonCapalbo2008},
density~\cite{AlonEtAlDensity}, or degeneracy~\cite{AllenBottcherLiebenau2026}.

Random graphs also provide universal hosts, notably for bounded-degree
spanning graphs~\cite{DellamonicaEtAl2015,FerberNenadov2018,FerberKronenbergLuh2019}
and trees~\cite{Montgomery2019,CohenAntonirEtAl2026}. For example,
Montgomery~\cite{Montgomery2019} proved that $G(n,C_\Delta\log n/n)$
contains every $n$-vertex tree of maximum degree at most $\Delta$ with
high probability, at the optimal order of edge density. Here $G(n,p)$
is the Erd\H{o}s-R\'enyi random graph with edge probability $p$, and \emph{with high probability}
means with probability tending to one.
The \emph{randomly perturbed} model adds sparse random edges to a graph
of linear minimum degree~\cite{BohmanFriezeMartin2003}; results include
prescribed spanning embeddings~\cite{KrivelevichKwanSudakov2017,BottcherEtAl2020Embedding}
and universality for bounded-degree spanning trees~\cite{BottcherEtAl2019Universality}
and graphs of maximum degree two~\cite{Parczyk2020,BottcherEtAl2024Square}.

We study the corresponding questions for permutations, where containment
must respect two linear orders. Write $[n]=\{1,\ldots,n\}$ and let
$\mathfrak S_n$ denote the set of permutations of $[n]$. A permutation
$\sigma\in\mathfrak S_N$ \emph{contains} $\pi\in\mathfrak S_n$ as a
\emph{pattern} if there are indices $1\le x_1<\cdots<x_n\le N$ such that
\[
  \sigma(x_i)<\sigma(x_j)
  \quad\Longleftrightarrow\quad
  \pi(i)<\pi(j)\qquad(i,j\in[n]).
\]
The indices need not be consecutive. We call $\sigma$ \emph{$n$-universal},
or an \emph{$n$-superpattern}, if it contains every member of
$\mathfrak S_n$. As in the graph setting, one may ask for the shortest
universal permutation or for the length at which a uniformly random
permutation becomes universal with high probability.

Universality is closely related to pattern avoidance. The Stanley--Wilf
conjecture asserts that, for each fixed pattern $\pi$, the number
$a_N(\pi)$ of permutations in $\mathfrak S_N$ avoiding $\pi$ is at most
$C_\pi^N$. Alon and Friedgut~\cite{AlonFriedgut2000} had earlier obtained
an almost exponential upper bound. Marcus and Tardos~\cite{MarcusTardos2004} proved it by
establishing the F\"uredi--Hajnal conjecture~\cite{FurediHajnal1992}:
an $m\times m$ zero--one matrix avoiding a fixed permutation matrix has
only $O_\pi(m)$ ones. The implication between the conjectures is due to
Klazar~\cite{Klazar2000}; Cibulka~\cite{Cibulka2009} later sharpened the
quantitative relation between their constants.
At most $\sum_{\pi\in\mathfrak S_n}a_N(\pi)$
permutations of length $N$ fail to be $n$-universal, so sufficiently
strong uniform avoidance estimates would imply random universality.
However, Fox~\cite{Fox} showed that the optimal exponential constants
can grow superpolynomially with the pattern length; these fixed-pattern
bounds therefore do not yield universality in the quadratic regime
considered here.

In the deterministic setting, let $u(n)$ be the minimum length of an
$n$-universal permutation. The counting bound
$\binom{u(n)}n\ge n!$ gives $u(n)\ge(e^{-2}+o(1))n^2$; see
Arratia~\cite{Arratia1999}. Miller~\cite{Miller2009} constructed universal
permutations of length $n(n+1)/2$, and Engen and Vatter~\cite{EngenVatter2021}
improved this to $\lceil(n^2+1)/2\rceil$. Chroman, Kwan and
Singhal~\cite{ChromanKwanSinghal2021} improved the lower bound to
$u(n)\ge(1.000076/e^2)n^2$ for all sufficiently large $n$. Thus the
deterministic problem has quadratic order, while its leading constant
remains undetermined.

For the random problem, let $\sigma_N$ be uniformly distributed in
$\mathfrak S_N$. The classical results of Logan and
Shepp~\cite{LoganShepp1977} and Vershik and Kerov~\cite{VershikKerov1977}
show that its longest increasing subsequence has length
$(2+o(1))\sqrt N$ with high probability; see also Romik's
book~\cite{Romik2015}. Consequently, for every fixed
$\varepsilon\in(0,1)$, a random permutation of length
$\lfloor(1-\varepsilon)n^2/4\rfloor$ is not $n$-universal with high
probability. Alon conjectured that this obstruction determines the
threshold: for every fixed $\varepsilon>0$, a random permutation of length
$\lceil(1+\varepsilon)n^2/4\rceil$ is $n$-universal with high probability;
see~\cite{Arratia1999,Alon2016}. He also observed that $O(n^2\log n)$
suffices and posed the weaker problem of proving that $1000n^2$ suffices
\cite{Alon2016}.
This problem is also listed on Xiaoyu He's homepage~\cite{HeProblems}.

He and Kwan~\cite[Theorem~1.2]{He-Kwan} proved that a random permutation
of length $\lceil2000n^2\log\log n\rceil$ is $n$-universal with high
probability. They also showed that a random permutation of length $20n^2$
typically contains every member of a fixed family
$\mathcal Q_n\subseteq\mathfrak S_n$ of size $(1-o(1))n!$
\cite[Theorem~1.3]{He-Kwan}. We establish the quadratic bound for full
universality, with a quantitative estimate on the failure probability.

\begin{theorem}\label{thm:main}
   There are absolute constants $c>0$ and $n_0$ such that, for every integer
$n\ge n_0$, a uniformly random permutation $\sigma\in\mathfrak{S}_{32n^2}$
satisfies
\[
  \Pr\bigl(\sigma\text{ is not }n\text{-universal}\bigr)
  \le \exp\!\left(-c\,\frac n{\log n}\right).
\]
\end{theorem}

\Cref{thm:main} removes the $\log\log n$ factor from the bound of He and
Kwan and resolves Alon's weaker $1000n^2$ problem. Together with the
increasing-subsequence obstruction, it determines the order of magnitude
of the length required for random universality. The constant $32$ is
chosen for convenience; the sharp constant $1/4$ predicted by Alon's
conjecture remains open. In~\Cref{sec:concluding-remarks}, we sketch how to reduce the constant $32$ to any value greater than $1$ and explain why $1$ is a natural barrier for this approach. 
Regarding the non-universal probability, the bound $\exp(-\Omega(n/\log n))$ comes from the limitation of the approach and is likely not sharp. We discuss that also in~\Cref{sec:concluding-remarks}.

Our proof uses the interval-minor coupling of He and Kwan to pass to an
independent zero--one matrix. We split each target permutation into tuples
of length $O(\log n)$ and search for them in successive column strips, restarting an
unfinished tuple in the next strip. Bounding the overlap between the
supports of these tests gives uniform control over all short tuples.
A deterministic counting argument then yields every target permutation,
without a union bound over all $n!$ patterns.

\section{Interval minors}
Let $P_\pi$ be the permutation matrix
whose ones are in positions $(\pi(i),i)$. Following Fox~\cite{Fox}, an
\emph{interval contraction} replaces two consecutive rows or columns of a
zero--one matrix by their entrywise Boolean OR. A matrix $B$ is an
\emph{interval minor} of $A$ if it can be obtained from $A$ by interval
contractions and changing ones to zeros. This relation is transitive, and
$\sigma$ contains $\pi$ as a pattern if and only if $P_\pi$ is an interval
minor of $P_\sigma$; see~\cite{He-Kwan}. We call $A$
\emph{$n$-universal} if every $P_\pi$, $\pi\in\mathfrak S_n$, is an
interval minor of $A$.

\begin{lemma}[{\cite[Lemma~2.2]{He-Kwan}}]\label{lem:coupling}
Let $N$ be divisible by $4n$. A uniform random permutation
$\sigma\in\mathfrak S_N$ and a $2n\times N/(4n)$ matrix $A$ with
independent Bernoulli$(1/2)$ entries can be coupled so that $A$ is always
an interval minor of $P_\sigma$.
\end{lemma}

\section{Proof of \texorpdfstring{\Cref{thm:main}}{the main theorem}}

\begin{proof}[Proof of \Cref{thm:main}]
Let $n$ be sufficiently large, and put
\[
  b=\lceil48\log n\rceil,\qquad r=\lceil n/b\rceil.
\]
By \Cref{lem:coupling} with $N=32n^2$ and transitivity of interval minors,
it suffices to prove that a $2n\times8n$ matrix $A$ with independent
Bernoulli$(1/2)$ entries is $n$-universal with the required probability.
Since $6br<6(n+b)\le8n$ for large $n$, we may use the disjoint column
intervals
\[
  J_j=\{3b(j-1)+1,\ldots,3bj\},\qquad j\in[2r].
\]

Let $\mathcal W$ consist of all ordered tuples
$w=(v_1,\ldots,v_\ell)$ of distinct elements of $[n]$, with
$1\le\ell\le b$. For each $w$, let $B_w\subseteq[n]\times[2r]$ be the
set of pairs $(t,j)$ for which there are no columns
$x_1<\cdots<x_\ell$ in $J_j$ satisfying
\begin{equation}\label{eq:tuple-embedding}
  A(t+v_i,x_i)=1\qquad(i\in[\ell]).
\end{equation}
Call such pairs bad tests, and define
\[
  \mathcal E=\{|B_w|<n\text{ for every }w\in\mathcal W\}.
\]

We first show deterministically that $\mathcal E$ implies universality.
Fix $\pi\in\mathfrak S_n$, and split its sequence of values into $r$
consecutive tuples $w_1,\ldots,w_r$, each of length $b$ except possibly
the last. These retain their original labels in $[n]$. For each shift
$t\in[n]$, scan the strips $J_1,\ldots,J_{2r}$ in order, attempting the
next unfinished tuple in each strip. On success, retain an embedding as in
\eqref{eq:tuple-embedding} and advance to the next tuple; on failure,
discard any partial embedding and retry the same tuple in the next strip.
A completed search gives columns $x_1<\cdots<x_n$ with
$A(t+\pi(i),x_i)=1$ for all $i$, and hence an interval minor $P_\pi$.

If all $n$ searches fail, each completes at most $r-1$ tuples and thus
encounters at least $r+1$ bad tests. There are at least $n(r+1)$ failures
in total. For each $i$, the failures while searching for $w_i$ correspond
to distinct pairs in $B_{w_i}$, since each shift visits each strip only
once. On $\mathcal E$ their total number is therefore at most
$\sum_{i=1}^r|B_{w_i}|<rn$, a contradiction. Thus $A$ is $n$-universal
on $\mathcal E$.

We now bound $\Pr(\mathcal E^c)$. Fix
$w=(v_1,\ldots,v_\ell)\in\mathcal W$ and a test $(t,j)$. The greedy
scan for \eqref{eq:tuple-embedding} queries the next required row in each
successive column, advancing whenever it sees a one; after completing the
tuple, it queries a fixed row in the remaining columns. Every query is a
fresh fair bit, since its column has not been inspected before. Thus the
failure probability is $\Pr(\Bin(3b,1/2)<\ell)$, and the Chernoff bound
\cite[Theorem~2.1]{JLR2000} gives
\begin{equation}\label{eq:single-test}
  \Pr((t,j)\in B_w)
  \le \Pr(\Bin(3b,1/2)<b)\le e^{-b/12}.
\end{equation}

Put $S=\{v_1,\ldots,v_\ell\}$. The bad-test event for $(t,j)$ depends
only on the entries in the deterministic set $(S+t)\times J_j$.
Join two tests when these sets intersect. Intersecting tests must lie in
the same strip and their shifts differ by an element of $S-S$, so this
graph has maximum degree at most
\[
  |S-S|-1\le\ell(\ell-1)\le b^2-1.
\]
Put $k=\lceil n/b^2\rceil$. The standard greedy bound for independent
sets implies that any $n$ tests contain at least $k$ tests with pairwise
disjoint supports. For any fixed set of $k$ such tests, their badness
events are mutually independent, so \eqref{eq:single-test} bounds the
probability that all are bad by $e^{-bk/12}$. Taking a union bound over
the possible sets of $k$ such tests yields
\[
  \Pr(|B_w|\ge n)\le(2nr)^k e^{-bk/12}.
\]
Finally, $|\mathcal W|\le\sum_{\ell=1}^b n^\ell\le bn^b$. Since
$2nr\le n^2$ for large $n$, we have
$\log(2nr)\le2\log n\le b/24$. Another union bound therefore gives
\begin{align*}
  \Pr(A\text{ is not }n\text{-universal})
  &\le\Pr(\mathcal E^c)
   \le bn^b(2nr)^k e^{-bk/12}\\
  &\le\exp\!\left(\log b+b\log n-\frac{bk}{24}\right)\\
  &\le\exp\!\left(O((\log n)^2)-\frac n{24b}\right)
   \le\exp\!\left(-c\,\frac n{\log n}\right)
\end{align*}
for an absolute $c>0$ and all sufficiently large $n$, as required.
\end{proof}

\section{Concluding remarks}\label{sec:concluding-remarks}
\subsection{Reducing the length of the random permutation}
We established that a random permutation of length $32n^2$ is $n$-universal with high probability.
The constant $32$ can be reduced to any constant greater than $1$
within the interval-minor framework. Here is a brief sketch of how to obtain this result.

Fix $\varepsilon>0$ and choose
$\delta\in(0,1)$ with $(1+\delta)^3<1+\varepsilon$. Choose a sufficiently
large constant $C_\delta$, set $b=\lceil C_\delta\log n\rceil$ and
$r=\lceil n/b\rceil$, and put
\[
  R=n+\lceil\delta n\rceil,\qquad
  T=\lceil(1+\delta)r\rceil,\qquad L=\lceil(1+\delta)Rb\rceil.
\]
Take $LT$ independent uniform variables $U_i\in[0,1]$, partition
$[0,1]$ into $R$ equal intervals $I_a$, and set
$A(a,i)=\mathbf1_{\{U_i\in I_a\}}$. After deleting empty rows, $A$
is an interval minor of the induced permutation matrix. Partition its
columns into $T$ strips of length $L$, and use the same tuple tests with
$s=R-n$ shifts in the original band labels.

The entries of $A$ are negatively associated: each column is a
multinomial vector with one trial, and independent unions preserve
negative association~\cite{JoagDevProschan1983}. Since bad-test events
are decreasing, tests with disjoint supports satisfy the same product
upper bound as independent tests. Moreover, a greedy test fails with
probability at most
$\Pr(\Bin(L,1/R)<b)\le e^{-a_\delta b}$ by the Chernoff bound, where
$a_\delta=\delta^2/(2(1+\delta))$. If a tuple has at least $\delta s$
bad tests, the same support-graph argument extracts
$k_\delta=\lceil\delta s/b^2\rceil$ of them with disjoint supports.
The union bounds over these sets and all tuples therefore give
\[
  \Pr(\exists w:\ |B_w|\ge\delta s)
  \le bn^b(sT)^{k_\delta}e^{-a_\delta b k_\delta}
  \le e^{-\Omega_\delta(n/\log n)},
\]
provided $C_\delta$ is sufficiently large. Indeed,
$\log(sT)\le2\log n+O_\delta(1)$, while
$bk_\delta\ge\delta s/b=\Omega_\delta(n/\log n)$ and
$\log(bn^b)=O_\delta((\log n)^2)$.
Outside this exceptional event, every tuple has fewer than $\delta s$
bad tests. For any target, if all $s$ searches
failed, they would encounter at least $s(T-r+1)>\delta sr$ bad tests,
whereas its $r$ tuples account for fewer than $\delta sr$ bad tests.
Thus every target is found.

Finally, $LT=(1+\delta)^3n^2+O_\delta(n\log n)<N=\lceil(1+\varepsilon)n^2\rceil$
for large $n$. The pattern induced by the first $LT$ entries of a uniform
permutation of length $N$ is itself uniform in $\mathfrak S_{LT}$.
Hence a uniform permutation of length $N$ is $n$-universal with failure
probability $e^{-\Omega_\varepsilon(n/\log n)}$.

Nevertheless, the leading constant $1$ is a natural barrier for this approach because
each search prescribes one of $R\ge n$ bands for every target entry.
Even without strip boundaries or restarts, a fixed-shift greedy search
finishes with probability $\Pr(\Bin(N,1/R)\ge n)$. If
$N\le(1-\eta)n^2$ for fixed $\eta\in(0,1)$, this is
$e^{-\Omega_\eta(n)}$ by the Chernoff bound. A union bound over the
$O(n)$ available shifts therefore shows that these searches typically
miss even a fixed target. Thus, improving the leading constant below $1$
requires new ideas.

\subsection{Sharper failure probability}
Sharper failure estimates are suggested by the following conjecture of
He and Kwan.
\begin{conjecture}[{\cite[Conjecture~5.1]{He-Kwan}}]\label{conj:containment}
Let $\sigma$ be a uniform random permutation of length $1000n^2$.
Then, uniformly over $\pi\in\mathfrak S_n$,
\[
  \Pr(\sigma\text{ does not contain }\pi)
  \le \exp(-n^{3/2+o(1)}).
\]
\end{conjecture}
By a union bound, this would also give failure probability
$\exp(-n^{3/2+o(1)})$ for $n$-universality at length $1000n^2$. If true, the exponent $3/2$ is sharp, as shown by Fox~\cite{Fox}.

Our current estimates do not reach Conjecture~\ref{conj:containment}. For a general
tuple length $b$, the overlap bound retains only $\lceil n/b^2\rceil$
disjoint tests, giving a probability factor $e^{-\Theta(n/b)}$ when
$b=o(\sqrt n)$. Counting their possible locations costs
$\exp(O((n/b^2)\log n))$, and counting tuples costs
$\exp(O(b\log n))$. Optimizing these bounds takes $b$ of order $\log n$
and yields an exponent of order $n/\log n$. 

Moreover, the sufficient event $\mathcal E$ itself cannot yield a
super-linear failure exponent. If the rows $b+1,2b+1,\ldots,rb+1$ are
zero throughout $J_1$, every shift of $w=(1,\ldots,b)$ fails in that
strip: each interval $t+[b]$, $t\in[n]$, contains one of these rows.
This forces $|B_w|\ge n$ and has probability $2^{-3br}$. Hence
\[
  \Pr(\mathcal E^c)\ge 2^{-3br}\ge 2^{-6n},
\]
where $br<n+b\le2n$. Thus even a sharper estimate of
$\Pr(\mathcal E^c)$ cannot give the exponent $n^{3/2+o(1)}$ in Conjecture~\ref{conj:containment}. This local obstruction need not prevent containment; reaching that scale requires a different sufficient event
or more direct control of containment failure.

\bibliographystyle{plain}
\bibliography{ref}

@article{He-Kwan,
  title={Universality of random permutations},
  author={He, Xiaoyu and Kwan, Matthew},
  journal={Bulletin of the London Mathematical Society},
  volume={52},
  number={3},
  pages={515--529},
  year={2020},
  publisher={Wiley Online Library},
  doi={10.1112/blms.12345}
}

@misc{HeProblems,
  author={He, Xiaoyu},
  title={Problems},
  howpublished={\url{https://alkjash.github.io/problems/}},
  note={Open problem 2. Accessed 9 October 2026}
}

@article{Fox,
  title={{Stanley--Wilf} limits are typically exponential},
  author={Fox, Jacob},
  journal={arXiv preprint arXiv:1310.8378},
  year={2013},
  url={https://arxiv.org/abs/1310.8378}
}

@article{MarcusTardos2004,
  author={Marcus, Adam and Tardos, G{\'a}bor},
  title={Excluded permutation matrices and the {Stanley--Wilf} conjecture},
  journal={Journal of Combinatorial Theory, Series A},
  volume={107},
  number={1},
  pages={153--160},
  year={2004},
  doi={10.1016/j.jcta.2004.04.002}
}

@article{FurediHajnal1992,
  author={F{\"u}redi, Zolt{\'a}n and Hajnal, P{\'e}ter},
  title={{Davenport--Schinzel} theory of matrices},
  journal={Discrete Mathematics},
  volume={103},
  number={3},
  pages={233--251},
  year={1992},
  doi={10.1016/0012-365X(92)90316-8}
}

@article{AlonFriedgut2000,
  author={Alon, Noga and Friedgut, Ehud},
  title={On the number of permutations avoiding a given pattern},
  journal={Journal of Combinatorial Theory, Series A},
  volume={89},
  number={1},
  pages={133--140},
  year={2000},
  doi={10.1006/jcta.1999.3002}
}

@incollection{Klazar2000,
  author={Klazar, Martin},
  title={The {F\"uredi--Hajnal} conjecture implies the {Stanley--Wilf} conjecture},
  booktitle={Formal Power Series and Algebraic Combinatorics},
  editor={Krob, Daniel and Mikhalev, Alexander A. and Mikhalev, Alexander V.},
  pages={250--255},
  publisher={Springer},
  address={Berlin, Heidelberg},
  year={2000},
  doi={10.1007/978-3-662-04166-6_22}
}

@article{Cibulka2009,
  author={Cibulka, Josef},
  title={On constants in the {F\"uredi--Hajnal} and the {Stanley--Wilf} conjecture},
  journal={Journal of Combinatorial Theory, Series A},
  volume={116},
  number={2},
  pages={290--302},
  year={2009},
  doi={10.1016/j.jcta.2008.06.003}
}

@book{JLR2000,
  title={Random Graphs},
  author={Janson, Svante and {\L}uczak, Tomasz and Ruci{\'n}ski, Andrzej},
  series={Wiley-Interscience Series in Discrete Mathematics and Optimization},
  publisher={John Wiley \& Sons},
  address={New York},
  year={2000},
  doi={10.1002/9781118032718}
}

@article{Rado1964,
  author={Rado, Richard},
  title={Universal graphs and universal functions},
  journal={Acta Arithmetica},
  volume={9},
  number={4},
  pages={331--340},
  year={1964},
  doi={10.4064/aa-9-4-331-340}
}

@article{Moon1965,
  author={Moon, J. W.},
  title={On minimal {$n$}-universal graphs},
  journal={Proceedings of the Glasgow Mathematical Association},
  volume={7},
  number={1},
  pages={32--33},
  year={1965},
  doi={10.1017/S2040618500035139}
}

@article{Alon2017,
  author={Alon, Noga},
  title={Asymptotically optimal induced universal graphs},
  journal={Geometric and Functional Analysis},
  volume={27},
  number={1},
  pages={1--32},
  year={2017},
  doi={10.1007/s00039-017-0396-9}
}

@inproceedings{AlonCapalbo2008,
  author={Alon, Noga and Capalbo, Michael},
  title={Optimal universal graphs with deterministic embedding},
  booktitle={Proceedings of the Nineteenth Annual ACM--SIAM Symposium on Discrete Algorithms},
  pages={373--378},
  publisher={ACM},
  year={2008}
}

@article{DellamonicaEtAl2015,
  author={Dellamonica, Jr., Domingos and Kohayakawa, Yoshiharu and R{\"o}dl, Vojt{\v e}ch and Ruci{\'n}ski, Andrzej},
  title={An improved upper bound on the density of universal random graphs},
  journal={Random Structures \& Algorithms},
  volume={46},
  number={2},
  pages={274--299},
  year={2015},
  doi={10.1002/rsa.20545}
}

@article{Montgomery2019,
  author={Montgomery, Richard},
  title={Spanning trees in random graphs},
  journal={Advances in Mathematics},
  volume={356},
  pages={106793},
  year={2019},
  doi={10.1016/j.aim.2019.106793}
}

@article{Arratia1999,
  author={Arratia, Richard},
  title={On the {Stanley--Wilf} conjecture for the number of permutations avoiding a given pattern},
  journal={Electronic Journal of Combinatorics},
  volume={6},
  number={1},
  pages={N1},
  year={1999},
  doi={10.37236/1477}
}

@article{Miller2009,
  author={Miller, Alison},
  title={Asymptotic bounds for permutations containing many different patterns},
  journal={Journal of Combinatorial Theory, Series A},
  volume={116},
  number={1},
  pages={92--108},
  year={2009},
  doi={10.1016/j.jcta.2008.04.007}
}

@article{EngenVatter2021,
  author={Engen, Michael and Vatter, Vincent},
  title={Containing all permutations},
  journal={American Mathematical Monthly},
  volume={128},
  number={1},
  pages={4--24},
  year={2021},
  doi={10.1080/00029890.2021.1835384}
}

@article{ChromanKwanSinghal2021,
  author={Chroman, Zachary and Kwan, Matthew and Singhal, Mihir},
  title={Lower bounds for superpatterns and universal sequences},
  journal={Journal of Combinatorial Theory, Series A},
  volume={182},
  pages={105467},
  year={2021},
  doi={10.1016/j.jcta.2021.105467}
}

@article{Alon2016,
  author={Alon, Noga},
  title={When are random permutations {$k$}-universal?},
  journal={Oberwolfach Reports},
  volume={13},
  number={2},
  pages={1218--1219},
  year={2016},
  note={In {\em Combinatorics and Probability}, Oberwolfach Report 22/2016},
  doi={10.4171/OWR/2016/22}
}

@book{Romik2015,
  author={Romik, Dan},
  title={The Surprising Mathematics of Longest Increasing Subsequences},
  publisher={Cambridge University Press},
  address={Cambridge},
  year={2015},
  doi={10.1017/CBO9781139872003}
}

@article{LoganShepp1977,
  author={Logan, B. F. and Shepp, L. A.},
  title={A variational problem for random {Young} tableaux},
  journal={Advances in Mathematics},
  volume={26},
  number={2},
  pages={206--222},
  year={1977},
  doi={10.1016/0001-8708(77)90030-5}
}

@article{VershikKerov1977,
  author={Vershik, A. M. and Kerov, S. V.},
  title={Asymptotic behavior of the {Plancherel} measure of the symmetric group and the limit form of {Young} tableaux},
  journal={Doklady Akademii Nauk SSSR},
  volume={233},
  number={6},
  pages={1024--1027},
  year={1977}
}

@article{AlonEtAlDensity,
  author={Alon, Noga and Dodson, Natalie and Jackson, Carmen and McCarty, Rose and Nenadov, Rajko and Southern, Lani},
  title={Universality for graphs with bounded density},
  journal={Journal of Combinatorial Theory, Series B},
  volume={178},
  pages={245--266},
  year={2026},
  doi={10.1016/j.jctb.2026.01.004},
  url={https://arxiv.org/abs/2311.05500v2}
}

@article{CohenAntonirEtAl2026,
  author={Cohen Antonir, Asaf and Lichev, Lyuben and Zhukovskii, Maksim},
  title={Universality in random graphs via optimal linking systems: trees and beyond},
  journal={arXiv preprint arXiv:2608.06358},
  year={2026},
  doi={10.48550/arXiv.2608.06358},
  url={https://arxiv.org/abs/2608.06358}
}

@article{AlonNenadov2019,
  author={Alon, Noga and Nenadov, Rajko},
  title={Optimal induced universal graphs for bounded-degree graphs},
  journal={Mathematical Proceedings of the Cambridge Philosophical Society},
  volume={166},
  number={1},
  pages={61--74},
  year={2019},
  doi={10.1017/S0305004117000706}
}

@article{EsperetJoretMorin2023,
  author={Esperet, Louis and Joret, Gwena{\"e}l and Morin, Pat},
  title={Sparse universal graphs for planarity},
  journal={Journal of the London Mathematical Society},
  volume={108},
  number={4},
  pages={1333--1357},
  year={2023},
  doi={10.1112/jlms.12781}
}

@article{AllenBottcherLiebenau2026,
  author={Allen, Peter and B{\"o}ttcher, Julia and Liebenau, Anita},
  title={Universality for graphs of bounded degeneracy},
  journal={Random Structures \& Algorithms},
  volume={68},
  number={2},
  pages={e70055},
  year={2026},
  doi={10.1002/rsa.70055}
}

@article{FerberNenadov2018,
  author={Ferber, Asaf and Nenadov, Rajko},
  title={Spanning universality in random graphs},
  journal={Random Structures \& Algorithms},
  volume={53},
  number={4},
  pages={604--637},
  year={2018},
  doi={10.1002/rsa.20816}
}

@article{FerberKronenbergLuh2019,
  author={Ferber, Asaf and Kronenberg, Gal and Luh, Kyle},
  title={Optimal threshold for a random graph to be {$2$}-universal},
  journal={Transactions of the American Mathematical Society},
  volume={372},
  number={6},
  pages={4239--4262},
  year={2019},
  doi={10.1090/tran/7727}
}

@article{BohmanFriezeMartin2003,
  author={Bohman, Tom and Frieze, Alan and Martin, Ryan},
  title={How many random edges make a dense graph {Hamiltonian}?},
  journal={Random Structures \& Algorithms},
  volume={22},
  number={1},
  pages={33--42},
  year={2003},
  doi={10.1002/rsa.10070}
}

@article{KrivelevichKwanSudakov2017,
  author={Krivelevich, Michael and Kwan, Matthew and Sudakov, Benny},
  title={Bounded-degree spanning trees in randomly perturbed graphs},
  journal={SIAM Journal on Discrete Mathematics},
  volume={31},
  number={1},
  pages={155--171},
  year={2017},
  doi={10.1137/15M1032910}
}

@article{BottcherEtAl2019Universality,
  author={B{\"o}ttcher, Julia and Han, Jie and Kohayakawa, Yoshiharu and Montgomery, Richard and Parczyk, Olaf and Person, Yury},
  title={Universality for bounded degree spanning trees in randomly perturbed graphs},
  journal={Random Structures \& Algorithms},
  volume={55},
  number={4},
  pages={854--864},
  year={2019},
  doi={10.1002/rsa.20850}
}

@article{BottcherEtAl2020Embedding,
  author={B{\"o}ttcher, Julia and Montgomery, Richard and Parczyk, Olaf and Person, Yury},
  title={Embedding spanning bounded degree graphs in randomly perturbed graphs},
  journal={Mathematika},
  volume={66},
  number={2},
  pages={422--447},
  year={2020},
  doi={10.1112/mtk.12005}
}

@article{Parczyk2020,
  author={Parczyk, Olaf},
  title={{2}-universality in randomly perturbed graphs},
  journal={European Journal of Combinatorics},
  volume={87},
  pages={103118},
  year={2020},
  doi={10.1016/j.ejc.2020.103118}
}

@article{BottcherEtAl2024Square,
  author={B{\"o}ttcher, Julia and Parczyk, Olaf and Sgueglia, Amedeo and Skokan, Jozef},
  title={The square of a {Hamilton} cycle in randomly perturbed graphs},
  journal={Random Structures \& Algorithms},
  volume={65},
  number={2},
  pages={342--386},
  year={2024},
  doi={10.1002/rsa.21215}
}

@article{JoagDevProschan1983,
  author={Joag-Dev, Kumar and Proschan, Frank},
  title={Negative Association of Random Variables with Applications},
  journal={The Annals of Statistics},
  volume={11},
  number={1},
  pages={286--295},
  year={1983},
  doi={10.1214/aos/1176346079}
}
\end{document}